\documentclass[10pt]{article}

\usepackage{amsthm}
\usepackage{amsmath}
\usepackage{amsfonts}
\usepackage{amssymb}
\usepackage{mathrsfs}
\usepackage{empheq}
\usepackage{stmaryrd}
\usepackage{dsfont}
\usepackage{enumerate}
\usepackage{cancel}
\usepackage{comment}
\usepackage{pgfplots}
\usepackage{soul}
\usepackage{multirow}
\usetikzlibrary{patterns} 
\newcommand{\hide}[1]{}
\usepackage[applemac]{inputenc}
\usepackage{hyperref}
\usepackage{xcolor}

\definecolor{mygreen}{rgb}{0, 0.5, 0}

\newcommand{\1}{\mathds{1}}

\newcommand{\N}{\mathbb{N}}

\newcommand{\R}{\mathbb{R}}
\newcommand{\C}{\mathbb{C}}

\newcommand{\Dr}{\mathscr{D}}

\newcommand{\Lr}{\mathscr{L}}

\newcommand{\Or}{\mathscr{Or}}

\newcommand{\vphi}{\varphi}
\newcommand{\eps}{\varepsilon}

\newcommand{\ovl}{\overline}

\newcommand{\vlim}{\lim\limits}

\newcommand{\vint}{\int\limits}

\newcommand{\tends}{\longrightarrow}

\newcommand{\wt}{\widetilde}

\newcommand{\loc}{\mathrm{loc}}

\renewcommand{\d}{\mathrm{d}}

\newcommand{\M}{\mathrm{max}}
\newcommand{\GN}{\mathrm{GN}}

\newcommand{\vi}{\mathrm{i}}
\newcommand{\e}{\mathrm{e}}

\renewcommand{\le}{\leqslant}
\renewcommand{\ge}{\geqslant}
\renewcommand{\Re}{\mathrm{Re}}
\renewcommand{\Im}{\mathrm{Im}}
\newcommand{\bs}{\boldsymbol}
\newcommand{\p}{\prime}

\newcommand{\eqdef}{\stackrel{\mathrm{def}}{=}}
\DeclareMathOperator{\supp}{supp}

\numberwithin{equation}{section}

\newtheorem{thm}{Theorem}[section]
\newtheorem{prop}[thm]{Proposition}

\newtheorem{lem}[thm]{Lemma}

\theoremstyle{definition}
\newtheorem{rmk}[thm]{Remark}
\newtheorem{defi}[thm]{Definition}

\newtheorem{assum}[thm]{Assumption}

\newenvironment{proof*}{\noindent{\bf Proof.}}{\qed}
\newenvironment{vproof}[1]{\noindent{\bf Proof #1}}{\qed}

\title{\huge \sc Damped nonlinear Ginzburg--Landau equation with saturation. Part II. Strong Stabilization}
\author{\sc Pascal Bégout$^*$ and Jes\'us Ildefonso D{\'{\i}}az$^\dagger$}
\date{}

\begin{document}

\maketitle

\begin{center}
\begin{tabular}{ll}
\hspace*{-.28cm}$^*$ Toulouse School of Economics	&	\hspace*{-.25cm}$^\dagger$ Instituto de Matem\'atica Interdisciplinar	\\
Université Toulouse Capitole 						&	Universidad Complutense de Madrid								\\
Institut de Mathématiques de Toulouse 				&	Plaza de las Ciencias, 3										\\
1, Esplanade de l'Université 						&	28040 Madrid, SPAIN										\\
31080 Toulouse Cedex 6, FRANCE					& 	\\
{\footnotesize E-mail\:: \href{mailto:Pascal.Begout@math.cnrs.fr}{\texttt{Pascal.Begout@math.cnrs.fr}}}
&
{\footnotesize E-mail\:: \href{mailto:jidiaz@ucm.es}{\texttt{jidiaz@ucm.es}}}
\end{tabular}
\end{center}

\begin{abstract}
We study the complex Ginzburg--Landau equation posed on possibly unbounded domains, including some singular and saturated nonlinear damping terms. This model interpolates between the nonlinear Schr\"odinger equation and dissipative parabolic dynamics through a complex time-derivative prefactor, capturing the interplay between dispersion and dissipation. As a continuation of our previous study on the existence and uniqueness of solutions, we prove here some strong stabilization properties. In particular, we show the finite time extinction of solutions induced by the nonlinear saturation mechanism, which, sometimes, can be understood as a bang-bang control. The analysis relies on refined energy methods. Our results provide a rigorous justification of nonlinear dissipation as an effective stabilization mechanism for this class of complex equations where the maximum principle fails.
\end{abstract}

{\let\thefootnote\relax\footnotetext{Pascal Bégout acknowledges funding from ANR under grant ANR-17-EUR-0010 (Investissements d'Avenir program)}}
{\let\thefootnote\relax\footnotetext{The research of J.\,I.\:D\'{\i}az was partially supported by the project PID-2020-112517GBI00 of the AEI and MCIU/AEI/10.13039/-501100011033/FEDER, EU}}
{\let\thefootnote\relax\footnotetext{$^*$\href{https://orcid.org/0000-0002-9172-3057}{https://orcid.org/0000-0002-9172-3057}}}
{\let\thefootnote\relax\footnotetext{$^\dagger$\href{https://orcid.org/0000-0003-1730-9509}{https://orcid.org/0000-0003-1730-9509}}}
{\let\thefootnote\relax}
{\let\thefootnote\relax}

\tableofcontents

\baselineskip .6cm

\section{Introduction}
\label{introduction}

The complex Ginzburg--Landau equation constitutes one of the most fundamental models in the theory of nonlinear dissipative systems. For a more detailed introduction to the model we will consider in this paper we send the reader to the Part I of our study (see \cite{Part1}).

\bigskip
\noindent
In several recent works (see, e.g., \cite{MR4725781}), the strong stabilization of a damped nonlinear Schr\"{o}dinger equation with saturation effects was established on unbounded domains. That analysis demonstrates that suitably chosen nonlinear damping mechanisms can overcome dispersive
effects even in the absence of compactness properties typically available in bounded domains. Such results are particularly relevant for physical systems
modeled in open space, where boundary confinement cannot be assumed. The main goal of this paper is to extend the general approach taken in the
theory presented in \cite{MR4725781} in order to extend previous results in the literature on complex Ginzburg--Landau equation in which the
saturation term is understood as an absorption term (see, e.g., Antontsev, Dias and Figueira~\cite{MR3208711} and~\cite{zbMATH05502727,MR2194979,MR2268809,MR2365580,MR4098331}).

\bigskip
\noindent
The damped nonlinear Schr\"{o}dinger equation may be viewed as a limiting or simplified model within the broader Ginzburg--Landau framework. Introducing a complex coefficient in front of the time derivative allows one to interpolate continuously between purely dispersive Schr\"{o}dinger dynamics and purely dissipative parabolic dynamics. This observation motivates the extension of the stabilization theory developed in \cite{MR4725781} to the complex Ginzburg--Landau equation posed on general domains $\Omega\subseteq\R^N$ (possible unbounded), with boundary $\partial\Omega,$
\begin{empheq}[left=\empheqlbrace]{align}
\label{gl}
\e^{-\vi\theta}\frac{\partial u}{\partial t}-\Delta u+a|u|^{-(1-m)}u+b|u|^{p-1}u+\gamma u=f,	&	\text{ in } (0,\infty)\times\Omega,				\\
\label{glb}
u_{|\partial\Omega}=0,													&	\text{ on } (0,\infty)\times\partial\Omega,	\dfrac{}{}	\\
\label{u0}
u(0)= u_0,																&	\text{ in } \Omega,
\end{empheq}
where $\theta\in\left(-\frac\pi2,\frac\pi2\right),$ $0\le m\le1$ and $a,b,\gamma\in\C.$ Here we write, for generality $p\in(1,\infty)$ but the physically more often case considered in the literature corresponds to $p=3.$

\bigskip
\noindent
For $\theta\in\left[-\frac\pi2,\frac\pi2\right]$ and $m\ge0,$ we introduce the following set of complex numbers:
\begin{gather*}
C_\theta(m)=\Big\{z\in\C; \; \Re(z\e^{\vi\theta})>0 \text{ and } 2\sqrt m\,\Re(z\e^{\vi\theta})\ge|1-m|\,|\Im(z\e^{\vi\theta})|\Big\}.
\end{gather*}
In the particular cases in which $m\in\{0,1\},$ the set $C_\theta(m)$ becomes,
\begin{align*}
& C_\theta(0)=\Big\{z\in\C; \; \Re(z\e^{\vi\theta})>0 \text{ and } \Im(z\e^{\vi\theta})=0\Big\},	\\
& C_\theta(1)=\Big\{z\in\C; \; \Re(z\e^{\vi\theta})>0\Big\},
\end{align*}
and actually,
\begin{gather*}
C_\theta(0)=\Big\{z\in\C; \; \exists\mu>0 \text{ such that } z=\mu\,\e^{-\vi\theta}\Big\}.
\end{gather*}

\noindent
We note that if $\theta=\dfrac\pi2,$ $0\le m\le1,$ $a\in C_\theta(m),$ $b=0,$ $\gamma=-V(x)\in L^1_\loc(\Omega;\R)$ and $f\in L^1_\loc\big([0,\infty);L^2(\Omega)\big),$ then equation~\eqref{gl} becomes
\begin{gather}
\label{nls}
\vi\frac{\partial u}{\partial t}+\Delta u+V(x)u-a|u|^{-(1-m)}u=-f.
\end{gather}
It follows that the nonlinear Schr\"{o}dinger equation~\eqref{nls} is a limit case of the Ginzburg--Landau equation~\eqref{gl}, in term of $\theta.$ But the Ginzburg--Landau equation~\eqref{gl} may also be considered as an intermediate equation between the nonlinear Schr\"{o}dinger equation and the nonlinear heat equation
\begin{gather*}
\frac{\partial u}{\partial t}-\Delta u+a|u|^{-(1-m)}u=f,
\end{gather*}
by taking $\theta=0,$ $a\in\R$ and $b=\gamma=0$ in \eqref{gl}. In this last case, $a\in C_0(m)$ only means that $a$ is a positive real number.

\bigskip
\noindent
The strategy of the proofs in this paper relies on the use of suitable energy methods, sharpening the ones presented in the monograph \cite{MR2002i35001}. Those methods capture the effective dissipation induced by the nonlinear terms, combined with refined energy estimates adapted to unbounded domains. Singular nonlinearities with $0\le m<1$ require weak formulations and truncation arguments to control the dynamics near vanishing amplitudes.

\bigskip
\noindent
In stabilization problems, the presence of a damped saturated term plays a crucial role. Linear mechanisms and Lipschitz nonlinear terms alone produce only exponential decay but not strong stabilization properties.

\bigskip
\noindent
From a physical standpoint, the stabilization mechanism analyzed in this work can be interpreted as an effective dissipation process capable of counterbalancing dispersion and diffusion in open systems. In unbounded spatial domains, energy injected locally can escape to infinity through wave propagation or diffusive transport, preventing the formation of confined modes and undermining stabilization mechanisms based solely on linear damping.

\bigskip
\noindent
The nonlinear terms appearing in the complex Ginzburg--Landau equation introduce amplitude-dependent dissipation that becomes particularly effective in regimes where linear mechanisms fail. The singular term $|u|^{-(1-m)}u$ acts as a strong damping mechanism near low-amplitude states, suppressing residual oscillations and preventing the persistence of small-amplitude coherent structures. From the physical point of view, this term can be interpreted as a saturation or threshold effect that inhibits the survival of weak excitations.

\bigskip
\noindent
The complex prefactor $e^{-\vi\theta}$ plays a fundamental role in shaping the dynamics. For $\theta\neq0$, the system no longer conserves energy in the Hamiltonian sense, and the interaction between dispersive and dissipative components leads to a gradual relaxation toward equilibrium. This behavior
is characteristic of systems far from equilibrium, where dissipation and dispersion coexist and compete.

\bigskip
\noindent
From the perspective of nonlinear dynamics, the stabilization results obtained in this work indicate that the complex Ginzburg--Landau equation on
unbounded domains behaves as a genuinely dissipative system, despite the absence of geometric confinement and the presence of continuous spectrum.
The nonlinear damping mechanisms effectively restore asymptotic stability by suppressing long-wavelength excitations and dispersive tails, leading to
strong convergence toward stationary states.

\bigskip
\noindent
These results provide a rigorous mathematical justification for the physical intuition that nonlinear dissipation and saturation can stabilize extended
systems even in open geometries, a phenomenon observed in a variety of physical contexts ranging from superconductivity and nonlinear optics to
pattern-forming systems far from equilibrium.

\bigskip
\noindent
One of our main motivations is the rigorous proof of the strong stabilization (in a finite time) to $u=0.$ This qualitative property is also called in the literature as the Finite Time Extinction property and it is also related with the so-called Finite Time Null controllability in Control Theory. For instance, the case of a pure saturation $m=0$ nonlinearity, as the one considered in \eqref{gl} can be understood also in the framework of Control Theory as a special case of a feed-back control $y(t,x)$ of ``bang-bang type" for the complex Ginzburg--Landau equation when we write \eqref{gl} in the form
\begin{gather*}
e^{-\vi\theta}\frac{\partial u}{\partial t}-\Delta u+b|u|^{p-1}u+\gamma u=f+y(t,x),	\text{ in } (0,\infty)\times\Omega,
\end{gather*}
with
\begin{gather*}
y(t,x)=-\vi\mu\frac{u(t,x)}{|u(t,x)|}.
\end{gather*}
where $\mu>0.$ This type of control has been considered in the applications to many dissipative evolution equations (see \cite{MR4725781} and its references). Nevertheless, the controllability for the complex Ginzburg--Landau equation is more delicate (for some related results, see, e.g., Rosier and Zhang~\cite{MR2504047} and Fenza, Labbadi and Ouzahra~\cite{FenzaLabbadiOuzahra}).

\bigskip
\noindent
In this paper, finite time extinction property (finite stabilization) of the solutions are obtained under the assumption that $a\in C_\theta(m),$ while for the equation~\eqref{nls}, they are proved in the series of papers \cite{MR4098330,MR4053613,MR4340780,MR4503241,MR4725781} under the assumption that $-a\in C(m),$ where
\begin{gather*}
C(m)=\Big\{z\in\C; \; \Im(z)>0 \text{ and } 2\sqrt m\Im(z)\ge(1-m)|\Re(z)|\Big\}.
\end{gather*}
Finally, notice that $a\in C_\frac\pi2(m)$ if, and only if, $-a\in C(m).$

\bigskip
\noindent
The organization of this paper is the following. Section~\ref{finite} presents the statements of the main results concerning the strong stabilization. In Section ~\ref{proofext} we present the proofs of the results concerning  the strong stabilization of the solutions.

\bigskip
\noindent
We collect here some notations that will be used along with this paper. For $t\in\R,$ $t_+=\max\{t,0\}$ is the positive part of $t.$ Unless if specified, all functions are complex-valued and all the vector spaces are considered over the field $\R.$ For a Banach space $X,$ we denote by $X^\star\eqdef\Lr(X;\R)$ its topological dual and by $\langle\: . \; , \: . \:\rangle_{X^\star,X}\in\R$ the $X^\star-X$ duality product. The product $\vi T\in X^\star,$ for $T\in X^\star,$ is defined in~\cite{MR4521439}. For $1\le p\le\infty,$ $p^\prime$ is the conjugate of $p$ defined by $\frac{1}{p}+\frac{1}{p^\prime}=1.$ For a Banach space $X$ and $p\in(0,\infty],$ $u\in L^p_\loc\big([0,\infty);X\big)$ means that $u\in L^p_\loc\big((0,\infty);X\big)$ and for any $T>0,$ $u_{|(0,T)}\in L^p\big((0,T);X\big).$ In the same way, we will use the notation $u\in W^{1,p}_\loc\big([0,\infty);X\big).$ If $p\in(0,\infty]$ and $r=0$ then $L^\frac{p}r(\Omega)=L^\infty(\Omega)$ and $W^{1,\frac{p}r}(\Omega)=W^{1,\infty}(\Omega).$ Finally, we denote by $C$ auxiliary positive constants, and sometimes, for positive parameters $a_1,\ldots,a_n,$ write as $C(a_1,\ldots,a_n)$ to indicate that the constant $C$ depends only and continuously on $a_1,\ldots,a_n$ (we will use this convention for constants which are not denoted merely by ``$C$'').

\section{Finite time extinction property}
\label{finite}

Let us recall that if $\theta\in\left[-\frac\pi2,\frac\pi2\right]$ and $m\ge0$ then $C_\theta(m)$ is defined by
\begin{gather*}
C_\theta(m)=\Big\{z\in\C; \; \Re(z\e^{\vi\theta})>0 \text{ and } 2\sqrt m\,\Re(z\e^{\vi\theta})\ge|1-m|\,|\Im(z\e^{\vi\theta})|\Big\}.
\end{gather*}
In order to have existence of solution, we make the following assumptions.

\begin{assum}
\label{ass}
We assume the following.
\begin{gather*}
\Omega \text{ is any nonempty open subset of } \R^N,	\\
-\dfrac\pi2<\theta<\dfrac\pi2,						\\
m\in[0,1] \; \text{ and } \; p\in(1,\infty),				\\
a\in C_\theta(m) \; \text{ and } \; b\in C_\theta(p)\cup\{0\},	\\
\gamma\in\C \text{ with } \Re(\gamma\e^{\vi\theta})\ge0.
\end{gather*}
\end{assum}

\begin{defi}
\label{defsatsec}
Let $\Or\subseteq\R^N$ be an open subset and let $u\in L^1_\loc(\Or).$ A function $U$ is said to be a \textit{saturated section} associated to $u$ if $U\in L^\infty(\Or),$ $\|U\|_{L^\infty(\Or)}\le1$ and $U=\dfrac{u}{|u|},$ almost everywhere where $u\neq0.$
\end{defi}

\noindent
Now, let us recall the notion of solution.

\begin{defi}
\label{defsol}
Let Assumption~\ref{ass} be fulfilled, let $f\in L^1_\loc\big([0,\infty);L^2(\Omega)\big)$ and let $u_0\in L^2(\Omega).$ We shall say that $u$ is an $H^2$-\textit{solution} \textit{to} \eqref{gl}--\eqref{u0}, if $u$ satisfies the following properties.
\begin{enumerate}
\item
\label{defsol1}
We have that
\begin{gather}
\label{defsol11}
u\in L^{m+1}_\loc\big([0,\infty);H^1_0(\Omega)\cap X_{m,p}\big)
\cap W^{1,\frac{m+1}m}_\loc\big([0,\infty);L^2(\Omega)+X_{m,p}^\star\big),
\end{gather}
where $X_{m,p}=L^{m+1}(\Omega)\cap L^{p+1}(\Omega).$
\item
\label{defsol2}
For almost every $t>0,$ $\Delta u(t)\in L^2(\Omega).$
\item
\label{defsol3}
\begin{enumerate}
\item
\label{defsol3a}
If $m>0$ then $u$ satisfies~\eqref{gl} in $\Dr^\p\big((0,\infty)\times\Omega\big).$
\item
\label{defsol3b}
If $m=0$ then there exists a saturated section $U$ associated to $u$ such that the pair $(u,U)$ satisfies
\begin{gather}
\label{gl0}
\e^{-\vi\theta}\frac{\partial u}{\partial t}-\Delta u+a\,U+b|u|^{p-1}u+\gamma u=f,
\end{gather}
in $\Dr^\p\big((0,\infty)\times\Omega\big).$
\end{enumerate}
\item
\label{defsol4}
We have that $u(0)=u_0,$ in $L^2(\Omega).$
\end{enumerate}
We shall say that $u$ is an $L^2$-\textit{solution} or a \textit{weak solution to} \eqref{gl}--\eqref{u0} if there exists a pair,
\begin{gather}
\label{fn}
(u_n,f_n)_{n\in\N}\subset C\big([0,\infty);L^2(\Omega)\big)\times L^1_\loc\big([0,\infty);L^2(\Omega)\big),
\end{gather}
such that for any $n\in\N,$ $u_n$ is an $H^2$-solution to \eqref{gl}--\eqref{u0} where the right hand side of \eqref{gl} is $f_n,$ and if
\begin{gather}
\label{cv}
u_n\xrightarrow[n\to\infty]{C([0,T];L^2(\Omega))}u \; \text{ and } \; f_n\xrightarrow[n\to\infty]{L^1((0,T);L^2(\Omega))}f,
\end{gather}
for any $T>0.$ Sometimes, we shall write $(u,f),$ $(u,U)$ or $(u,U,f)$ to designate a solution with the obvious meanings.
\end{defi}

\noindent
We recall that under Assumptions~\ref{ass}, if $f\in L^1_\loc\big([0,\infty);L^2(\Omega)\big)$ then for any $u_0\in L^2(\Omega),$ there exists a unique weak solution to \eqref{gl}--\eqref{u0} (\cite[Theorem~2.8]{Part1}).

\begin{thm}[\textbf{Infinite time extinction property}]
\label{thm0w}
Let Assumption~$\ref{ass}$ be fulfilled, let $u_0\in L^2(\Omega),$ let $f\in L^1\big((0,\infty);L^2(\Omega)\big)$ and let $u$ be the unique weak solution to \eqref{gl}--\eqref{u0}. Then,
\begin{gather*}
\vlim_{t\nearrow\infty}\|u(t)\|_{L^2(\Omega)}=0.
\end{gather*}
\end{thm}

\begin{prop}[\textbf{Infinite time extinction property}]
\label{propm1}
Let Assumption~$\ref{ass}$ be fulfilled with $m=1,$ let $f\in L^1\big((0,\infty);L^2(\Omega)\big),$ let $u_0\in L^2(\Omega)$ and let $u$ be the unique weak solution to \eqref{gl}--\eqref{u0}. If $f=0$ almost everywhere on $(T_0,\infty),$ for some $T_0\ge0,$ then
\begin{gather*}
\|u(t)\|_{L^2(\Omega)}\le\|u(T_0)\|_{L^2(\Omega)}e^{-\Re(a\e^{\vi\theta})(t-T_0)}.
\end{gather*}
for any $t\ge T_0.$
\end{prop}

\noindent
In order to have finite time extinction of the solutions, we make the following assumptions.

\begin{assum}
\label{assext}
Let Assumption~\ref{ass} be fulfilled with $m<1,$ let $f\in L^1_\loc\big([0,\infty);L^2(\Omega)\big),$ let $u_0\in L^2(\Omega)$ and let $u$ be the unique weak solution $u$ to \eqref{gl}--\eqref{u0}. We assume that there exists a finite time $T_0\ge0$ such that
\begin{gather}
\label{assext1}
\text{for almost every } (t,x)\in(T_0,\infty)\times\Omega, \; f(t,x)=0.
\end{gather}
If $m=0$, then we may make a weaker hypothesis. Instead of~\eqref{assext1}, we may assume that
\begin{gather}
\label{assext2}
f\in L^\infty\big((T_0,\infty);L^\infty(\Omega)\big) \text{ and } \|f\|_{L^\infty((T_0,\infty));L^\infty(\Omega))}<\Re(a\e^{\vi\theta}).
\end{gather}
Finally, we set $\delta=\frac{(N+2)-m(N-2)}{N(1-m)+4}\in\left(\frac12,1\right)$ and $\lambda=2(1-\delta)=\frac{4(1-m)}{N(1-m)+4}.$
\end{assum}

\begin{thm}[\textbf{Finite time extinction property}]
\label{thmext}
Let Assumption~$\ref{assext}$ be fulfilled.
\begin{enumerate}
\item
\label{thmext1}
For any $t\ge T_0,$
\begin{gather}
\label{thmext11}
\|u(t)\|_{L^2(\Omega)}\le\left(\|u(T_0)\|_{L^2(\Omega)}^\frac{4(1-m)}{N(1-m)+4}
-\lambda MC_\GN^{-\frac4{N(1-m)+4}}(t-T_0)\right)_+^\frac{N(1-m)+4}{4(1-m))},
\end{gather}
where $C_\GN$ is given by~\eqref{GN} below and
\begin{gather}
\label{thmext12}
M=\min\left\{\cos\theta,\Re(a\e^{\vi\theta})-\|f\|_{L^\infty((T_0,\infty);L^\infty(\Omega))}\right\}.
\end{gather}
In particular,
\begin{gather}
\label{thmext13}
\forall t\ge T_\star, \; \|u(t)\|_{L^2(\Omega)}=0,
\end{gather}
where
\begin{gather}
\label{thmext14}
T_\star\le\frac{C_\GN^\frac4{N(1-m)+4}}{\lambda M}\|u(T_0)\|_{L^2(\Omega)}^\frac{4(1-m)}{N(1-m)+4}+T_0.
\end{gather}
\item
\label{thmext2}
There exists $\eps_\star=\eps_\star(m,N)$ satisfying the following property. If
\begin{gather}
\label{thmext21}
\begin{cases}
\|u_0\|_{L^2(\Omega)}^{2(1-\delta)}\le\eps_\star T_0,						\medskip \\
\|f(t)\|_{L^2(\Omega)}^2\le\eps_\star\big(T_0-t\big)_+^\frac{2\delta-1}{1-\delta},
\end{cases}
\end{gather}
for almost every $t>0,$ then~\eqref{thmext13} holds true with $T_\star=T_0.$
\end{enumerate}
\end{thm}

\begin{rmk}
\label{rmkM}
Here are some comments about Theorem~\ref{thmext}.
\begin{enumerate}
\item
\label{rmkM1}
If $f$ satisfies \eqref{assext1} then $\|f\|_{L^\infty((T_0,\infty);L^\infty(\Omega))}=0$ and \eqref{thmext12} reads as: $M=\min\left\{\cos\theta,\Re(a\e^{\vi\theta})\right\}.$
\item
\label{rmkM2}
We have that: $\frac{2\delta-1}{1-\delta}=\frac{N(1-m)+4m}{2(1-m)}.$
\end{enumerate}
\end{rmk}

\section{Proof of the finite time extinction property}
\label{proofext}

If $u$ is an $H^2$-strong solution then the map $t\longmapsto\|u(t)\|_{L^2(\Omega)}^2$ belongs to $W^{1,\infty}_\loc\big([0,\infty);\R\big)$ and
\begin{gather}
\label{L2}
\begin{split}
&\frac12\frac{\d}{\d t}\|u(t)\|_{L^2(\Omega)}^2+\cos\theta\|\nabla u(t)\|_{L^2(\Omega)}^2+\Re(a\e^{\vi\theta})\|u(t)\|_{L^{m+1}(\Omega)}^{m+1}	\\
+\;&\Re(b\e^{\vi\theta})\|u(t)\|_{L^{p+1}(\Omega)}^{p+1}+\Re(\gamma\e^{\vi\theta})\|u(t)\|_{L^2(\Omega)}^2=\Re\left(\e^{\vi\theta}\vint_{\Omega}f(t,x)\,\ovl{u(t,x)}\,\d x\right),
\end{split}
\end{gather}
for almost every $t>0$ (\cite[Theorem~2.14]{Part1}). The proof of Property~\ref{thmext1} of Theorem~\ref{thmext} (as well as Property~\ref{thmext2}) relies on the estimate of the time derivative of the mass \eqref{L2} to arrive at the estimate
\begin{gather}
\label{exp1}
y^\p(t)+\nu\,y(t)^\delta\le0,
\end{gather}
where $y(\,.\,)=\|u(\,.\,)\|_{L^2(\Omega)}^2,$ for some $\nu>0$ and $\delta\in(0,1).$ But \eqref{L2} does not hold for the weak solutions, as well as \eqref{exp1}. As a consequence, we first prove Property~\ref{thmext1} for the strong solutions and then proceed by density. The passage to the limit is possible with the help of the continuous dependance (see \eqref{propdep2} below) and the weak solutions are approached by strong solutions with the help of Lemma~\ref{lemden1} below. This proves the extinction of the solution in finite time. But the proof of Property~\ref{thmext2} of Theorem~\ref{thmext}, which permits us to choose at which time the solution vanishes, is more delicate. To this end, we use again the estimate of the time derivative of the mass \eqref{L2} and the assumption
\begin{gather}
\label{exp2}
\|f(t)\|_{L^2(\Omega)}^2\le\eps_\star\big(T_0-t\big)_+^\frac{2\delta-1}{1-\delta},
\end{gather}
for almost every $t>0.$ We then obtain \eqref{exp1} and we then apply \cite[Lemma~5.2]{MR4053613} to obtain the extinction of the solution at time $T_0.$ Again, we have to consider strong solutions. But the key assumption \eqref{exp2} cannot be obtained for a smooth sequence $(f_n)_{n\in\N}$
that approaches the external source $f.$ Rather, we first prove a more general result (Lemma~\ref{lemden2} below) than \cite[Lemma~5.2]{MR4053613}, which permits us to prove Property~\ref{thmext2} by density.

\begin{lem}
\label{lemden1}
Let $I$ be an interval $($not necessarily open$)$ with $-\infty\le\inf I<\sup I\le\infty,$ let $1\le p<\infty,$ let $X$ be a Banach space and let $f\in L^p_\loc(I;X).$ Then there exist $(f_n)_{n\in\N}\subset\Dr(I;X)$ and $g\in L^p_\loc(I;\R)$ such that,
\begin{gather}
\label{lemden11}
f_n\xrightarrow[n\to\infty]{L^p_\loc(I;X)}f,								\\
\label{lemden12}
\text{for a.e.}\, t\in I, \; f_n(t)\xrightarrow[n\to\infty]{X}f(t),					\\
\label{lemden13}
\text{for a.e.}\, t\in I \text{ and any } n\in\N, \; \|f_n(t)\|_X\le g(t),				\\
\label{lemden14}
\text{for any } n\in\N, \; \supp f_n\subset{\supp f+\ovl B\left(0,\frac1n\right)}.
\end{gather}
If, in addition, $f\in L^p(I;X)$ then
\begin{gather}
\label{lemden15}
f_n\xrightarrow[n\to\infty]{L^p(I;X)}f \text{ and } g\in L^p(I;\R).
\end{gather}
Finally, if for some $q\in[1,\infty]$ and a Banach space $Y,$ $f\in L^q(I;Y)$ then $(f_n)_{n\in\N}\subset\Dr(I;Y)$ and
\begin{gather}
\label{lemden16}
\|f_n\|_{L^q(I;Y)}\le\|f\|_{L^q(I;Y)},
\end{gather}
for any $n\in\N.$
\end{lem}

\begin{proof*}
Let $f\in L^p_\loc(I;X).$ Let for each $n\in\N,$ $I_n=\left(\inf I+\frac1n,\sup I-\frac1n\right).$ Finally, let $(\rho_n)_{n\in\N}$ be a sequence of mollifiers. Let us denote by $\wt f$ the extension of $f$ by $0$ outside $I.$ Let $n,j\in\N.$ We define $g_{n,j}=\rho_j\star(\wt f\1_{I_n})_{|I}.$ It is well-known that for any $n\in\N,$ $(g_{n,j})_{j>n}\subset\Dr(I;X)$ and $g_{n,j}\xrightarrow{L^p(I;X)}f\1_{I_n},$ as $j\to\infty.$ See, for instance, Droniou~\cite[Th\'{e}or\`{e}me~1.7.1, p.27]{droniou}. See also Brezis~\cite{MR2759829} (Proposition~4.18, p.106 and Theorem~4.22, p.109). It follows that there exists an increasing function $\psi:\N\tends\N$ such that for any $n\in\N,$ setting $f_n=g_{n,\psi(n)},$ we have $f_n-f\1_{I_n}\xrightarrow{L^p(I;X)}0,$ as $n\to\infty.$ By the partial converse of the dominated convergence theorem for vector-valued functions (Droniou~\cite[Th\'{e}or\`{e}me~1.3.4, p.16]{droniou}), we may assume, by renumbering the sequence if necessary, that for a.e.\,$t\in I,$ $f_n(t)-f\1_{I_n}(t)\xrightarrow{X}0,$ as $n\to\infty,$ and $\|f_n-f\1_{I_n}\|_X\le g\in L^p(I;\R),$ a.e.\,in $I$ and for any $n\in\N.$ Since for any compact interval $J\subset I$ and $t\in\overset{\circ}I$ (the interior of $I),$ there exists $n_0\in\N$ such that for any $n>n_0,$ $J\subset I_n$ and $t\in I_n,$ we easily conclude that $(f_n)_{n\in\N}\subset\Dr(I;X)$ satisfies \eqref{lemden11} and \eqref{lemden12} (and also \eqref{lemden15} if $f\in L^p(I;X)).$ Since for any $n\in\N,$ $\supp\rho_n=\ovl B\left(0,\frac1n\right),$ \eqref{lemden14} comes from a classical result of the convolution of two functions (Brezis~\cite[Proposition~4.18]{MR2759829}). Finally, if $f\in L^q(I;Y)$ for some $q\in[1,\infty]$ and a Banach space $Y,$ then we have that $(f_n)_{n\in\N}\subset\Dr(I;Y)$ and \eqref{lemden16} comes from Young's inequality for vector-valued functions (Droniou~\cite[Proposition~1.7.1, p.25]{droniou}).
\medskip
\end{proof*}

\begin{lem}
\label{lemden2}
Let $\alpha,\delta>0.$ Let $g\in L^1_\loc([0,\infty);\R)$ be a nonnegative function and let $t_0\ge0.$
\begin{enumerate}
\item
\label{lemden21}
For any $z_0\ge0,$ there exists a unique solution $z\in W^{1,1}_\loc\big([t_0,\infty);\R\big)$ to
\begin{gather}
\label{lemden211}
\begin{cases}
\forall t\ge t_0, \; z(t)\ge0,						\medskip \\
\text{for a.e.} \; t>t_0, \; z^\p(t)+\alpha z(t)^\delta=g(t),
\end{cases}
\end{gather}
such that
\begin{gather}
\label{lemden212}
z(t_0)=z_0.
\end{gather}
Let $g_1,g_2\in L^1_\loc([0,\infty);\R)$ be nonnegative functions and let $z_1,z_2\in W^{1,1}_\loc\big([t_0,\infty);\R\big)$ be solutions to
\begin{gather}
\label{lemden213}
\begin{cases}
\forall t\ge t_0, \; z_j(t)\ge0,						\medskip \\
\text{for a.e.} \; t>0, \; z_j^\p(t)+\alpha z_j(t)^\delta=g_j(t),
\end{cases}
\end{gather}
for $j\in\{1,2\}.$ Then,
\begin{gather}
\label{lemden214}
|z_1(t)-z_2(t)|\le|z_1(s)-z_2(s)|+\vint_s^t|g_1(\sigma)-g_2(\sigma)|\d\sigma,
\end{gather}
for any $t\ge s\ge t_0.$
\item
\label{lemden22}
Let $z\in W^{1,1}_\loc\big([t_0,\infty);\R\big)$ be a solution to \eqref{lemden211} and let $y\in W^{1,1}_\loc\big([t_0,\infty);\R\big)$ be a nonnegative solution to
\begin{gather}
\label{lemden221}
\text{for a.e.} \; t>t_0, \; y^\p(t)+\alpha y(t)^\delta\le g(t),
\end{gather}
If for some $t_\star\in[t_0,T),$ $y(t_\star)\le z(t_\star)$ then
\begin{gather}
\label{lemden222}
\forall t\ge t_\star, \; y(t)\le z(t).
\end{gather}
\end{enumerate}
\end{lem}

\begin{proof*}
Let $\alpha,\delta>0.$ Let $g\in L^1_\loc([0,\infty);\R)$ with $g\ge0,$ a.e.\,in $(0,\infty),$ and let $t_0\ge0.$ \\
\textbf{Proof of Property~\ref{lemden21}.}
Let $([t_0,T_\M),z)$ be any maximal solution to \eqref{lemden211} with $T_\M<\infty.$ Then for any $t\in[t_0,T_\M),$
\begin{gather}
\label{demlemden211}
\forall t\in[t_0,T_\M), \; 0\le z(t)\le z(t_0)+\vint_0^{T_\M}g(s)\d s<\infty.
\end{gather}
Now, let $z_0\ge0$ and set for a.e.\,$t>0$ and any $x\in\R,$ $f(t,x)=g(t)-\alpha\left(x\1_{[0,\infty)}(x)\right)^\delta.$ Then by Carath\'{e}odory's Theorem and Zorn's Lemma, there exist $t_0<T_\M\le\infty$ and a maximal solution $z\in W^{1,1}_\loc\big([t_0,T_\M);\R\big)$ to $z^\p=f(\:.\:,z),$ a.e.\,on $(t_0,T_\M),$ that is
\begin{gather}
\label{demlemden212}
\text{for a.e.} \; t\in(t_0,T_\M), \; z^\p(t)+\alpha\left(z(t)\1_{\{z(t)\ge0\}}(t)\right)^\delta=g(t),
\end{gather}
such that $z(t_0)=z_0.$ In addition, the following blow-up alternative holds true: if $T_\M<\infty$ then $\vlim_{t\nearrow T_\M}|z(t)|=\infty.$ Now, assume by contradiction that for some $t_1\in(t_0,T_\M),$ $z(t_1)<0.$ Then, since $z(t_0)\ge0,$ we obtain by continuity the existence of a $T\in[t_0,T_\M)$ and of a $\delta\in(0,T_\M-T)$ such that $z(T)=0$ and for any $t\in(T,T+\delta],$ $z(t)<0.$ It then follows from \eqref{demlemden212} that $z^\p\ge0,$ a.e.\,on $(T,T+\delta),$ so that $0=z(T)\le z(T+\delta)<0,$ a contradiction. It follows that,
\begin{gather*}
\forall t\in[t_0,T_\M), \; z(t)\ge0,
\end{gather*}
and by \eqref{demlemden211} and the blow-up alternate, we obtain $T_\M=\infty.$ As a consequence, any maximal solution to \eqref{lemden211} is global. Now, let $g_1,g_2$ and $z_1,z_2$ be as in the statement of the lemma. Let $z=z_1-z_2$ and $g=g_1-g_2.$ It follows that,
\begin{gather*}
\text{for a.e.} \; t>t_0, \; z^\p(t)+\alpha(z_1(t)^\delta-z_2(t)^\delta)=g(t).
\end{gather*}
Multiplying by $z,$ using that $s\longmapsto\alpha s^\delta$ in increasing over $[0,\infty)$ and integrating, we get that $z$ satisfies,
\begin{gather*}
\forall t\ge t_0, \; |z_1(t)-z_2(t)|\le|z_1(s)-z_2(s)|+\vint_s^t|g_1(\sigma)-g_2(\sigma)|\d\sigma.
\end{gather*}
In particular, this implies uniqueness of the solution, and Property~\ref{lemden21} is proved.
\\
\textbf{Proof of Property~\ref{lemden22}.}
Let the assumptions be fulfilled. If \eqref{lemden222} does not hold then since $y(t_\star)\le z(t_\star),$ we have by continuity that there exist $t_\star\le T_\star<\infty$ and $\eps>0$ such that $y(T_\star)=z(T_\star)$ and $y(t)>z(t),$ for any $t\in(T_\star,T_\star+\eps).$ This leads with \eqref{lemden211} and \eqref{lemden221} to $y^\p\le z^\p,$ almost everywhere on $(T_\star,T_\star+\eps).$ Integrating over $(T_\star,t)$ for $t\in(T_\star,T_\star+\eps),$ we obtain that $y(t)\le z(t),$ for any $t\in[T_\star,T_\star+\eps],$ a contradiction. Hence the result.
\medskip
\end{proof*}

\noindent
Let us recall the following Gagliardo-Nirenberg inequality (Gagliardo~\cite{MR0109295}, Nirenberg~\cite{MR0109940}). Let $\Omega$ be an open subset of $\R^N$ and let $0\le m\le1.$ Then there exists $C_\GN=C_\GN(m,N)$ such that for any $u\in H^1_0(\Omega)\cap L^{m+1}(\Omega),$
\begin{gather}
\label{GN}
\|u\|_{L^2(\Omega)}^\frac{(N+2)-m(N-2)}2\le C_\GN\|u\|_{L^{m+1}(\Omega)}^{m+1}\|\nabla u\|_{L^2(\Omega)}^\frac{N(1-m)}2.
\end{gather}
It follows that,
\begin{align*}
	&	\; \|u\|_{L^2(\Omega)}^\frac{(N+2)-m(N-2)}2											\\
  \le	&	\; C_\GN\left(\|\nabla u\|_{L^2(\Omega)}^2+\|u\|_{L^{m+1}(\Omega)}^{m+1}\right)
				\left(\|\nabla u\|_{L^2(\Omega)}^2+\|u\|_{L^{m+1}(\Omega)}^{m+1}\right)^\frac{N(1-m)}4	\\
   =	&	\; C_\GN\left(\|\nabla u\|_{L^2(\Omega)}^2+\|u\|_{L^{m+1}(\Omega)}^{m+1}\right)^\frac{N(1-m)+4}4,
\end{align*}
and then
\begin{gather}
\label{demthmext0}
\|u\|_{L^2(\Omega)}^{2\delta}\le C_\GN^\frac4{N(1-m)+4}\left(\|\nabla u\|_{L^2(\Omega)}^2+\|u\|_{L^{m+1}(\Omega)}^{m+1}\right),
\end{gather}
for any $u\in H^1_0(\Omega)\cap L^{m+1}(\Omega),$ where $\delta$ is defined in Assumption~\ref{assext}. Finally, let us recall that if $(u,f)$ and $(\wt u, \wt f)$ are (strong or weak) solutions to \eqref{gl}--\eqref{glb} then
\begin{gather}
\label{propdep2}
\|u(t)-\wt u(t)\|_{L^2(\Omega)}\le\|u(s)-\wt u(s)\|_{L^2(\Omega)}+\vint_s^t\|f(\sigma)-\wt f(\sigma)\|_{L^2(\Omega)}\d\sigma,
\end{gather}
for any $t\ge s\ge0$ (\cite[Proposition~2.6]{Part1}). Now, we are able to prove Theorem~\ref{thmext}.

\medskip
\begin{vproof}{of Theorem~\ref{thmext}.}
Let $f,$ $u_0,$ $u$ and $M$ be as in the statement of the theorem and set for any $t\ge0,$ $y(t)=\|u(t)\|_{L^2(\Omega)}^2.$
\\
\textbf{Proof of Property~\ref{thmext1}.}
We only show that $u$ satisfies \eqref{thmext11}, from which \eqref{thmext13} and \eqref{thmext14} will follow. We first assume that $f\in\Dr((0,\infty);L^2(\Omega)\big)$ and $u_0\in\Dr(\Omega),$ so that $u$ is an $H^2$-solution (\cite[Theorem~2.14]{Part1}). We have by \eqref{L2}, \eqref{demthmext0}, \eqref{assext1} and  \eqref{assext2} that for a.e.\,$t>T_0,$
\begin{gather}
\label{demthmext1}
y^\p(t)+2\,\alpha\,y(t)^\delta\le0,
\end{gather}
where $\alpha=MC_\GN^{-\frac4{N(1-m)+4}}.$ After integration, we obtain that for any $t\ge T_0,$
\begin{gather}
\label{demthmext2}
\|u(t)\|_{L^2(\Omega)}^2
\le\left(\|u(T_0)\|_{L^2(\Omega)}^\frac{4(1-m)}{N(1-m)+4}-\lambda MC_\GN^{-\frac4{N(1-m)+4}}(t-T_0)\right)_+^\frac{N(1-m)+4}{2(1-m))},
\end{gather}
which is \eqref{thmext11}.
\\
\textbf{End of the proof when} $\bs{m>0.}$ Now, we consider the general case: $u_0\in L^2(\Omega)$ and $f\in L^1\big((0,\infty);L^2(\Omega)\big)$ which satisfies \eqref{assext1}. We apply Lemma~\ref{lemden1}: let $(\vphi_n)_{n\in\N}\subset\Dr(\Omega)$ and $(f_n)_{n\in\N}\subset\Dr\big((0,\infty);L^2(\Omega)\big)$ be such that, 
\begin{gather}
\label{demthmext3}
\vphi_n\xrightarrow[n\to\infty]{L^2(\Omega)}u_0 \; \text{ and } \; f_n\xrightarrow[n\to\infty]{L^1((0,\infty);L^2(\Omega))}f.
\end{gather}
Finally, for each $n\in\N,$ let $(u_n,f_n)$ be the $H^2$-solution to \eqref{gl}--\eqref{glb} such that $u_n(0)=\vphi_n.$ By \eqref{assext1}, \eqref{lemden14} and \eqref{demthmext2}, we get that for any $n\in\N$ and $t\ge T_0+\frac1n,$
\begin{gather}
\label{demthmext4}
\|u_n(t)\|_{L^2(\Omega)}
\le\left(\left\|u_n\left(T_0+\frac1n\right)\right\|_{L^2(\Omega)}^\frac{4(1-m)}{N(1-m)+4}
-\lambda MC_\GN^{-\frac4{N(1-m)+4}}\left(t-T_0-\frac1n\right)\right)_+^\frac{N(1-m)+4}{4(1-m))}.
\end{gather}
By \eqref{demthmext3} and \eqref{propdep2}, we may pass to the limit in \eqref{demthmext4}, so that $u$ satisfies \eqref{thmext11}.
\\
\textbf{End of the proof when} $\bs{m=0}.$
Assume that $u_0\in L^2(\Omega)$ and $f\in L^1_\loc\big([0,\infty);L^2(\Omega)\big)$ which satisfies \eqref{assext2}. By Lemma~\ref{lemden1}, there exist $(\vphi_n)_{n\in\N}\subset\Dr(\Omega)$ and $(f_n)_{n\in\N}\subset\Dr\big((0,\infty);L^2(\Omega)\big)$ such that, 
\begin{gather}
\label{demthmext5}
\vphi_n\xrightarrow[n\to\infty]{L^2(\Omega)}u(T_0) \text{ and for any } T>T_0, \; f_n\xrightarrow[n\to\infty]{L^1((T_0,T);L^2(\Omega))}f,	\\
\label{demthmext6}
\text{for any } n\in\N, \|f_n\|_{L^\infty((T_0,\infty);L^\infty(\Omega))}\le\|f\|_{L^\infty((T_0,\infty);L^\infty(\Omega))}.
\end{gather}
Let $n\in\N.$ Let $(u_n,f_n)$ be the $H^2$-solution to \eqref{gl}--\eqref{glb} such that $u_n(T_0)=\vphi_n$ (which is possible by uniqueness of the solution and the invariance of \eqref{gl} by time translation). By \eqref{demthmext6}, each $f_n$ satisfies \eqref{assext2}. Then by \eqref{demthmext2}, we have for any $n\in\N$ and $t\ge T_0,$
\begin{gather*}
\|u_n(t)\|_{L^2(\Omega)}^2
\le\left(\|u_n(T_0)\|_{L^2(\Omega)}^\frac{4(1-m)}{N(1-m)+4}-\lambda M_nC_\GN^{-\frac4{N(1-m)+4}}(t-T_0)\right)_+^\frac{N(1-m)+4}{2(1-m))},
\end{gather*}
where $M_n=\min\left\{\cos\theta,\Re(a\e^{\vi\theta})-\|f_n\|_{L^\infty((T_0,\infty);L^\infty(\Omega))}\right\}.$ By \eqref{demthmext6}, $M\le M_n,$ so that
\begin{gather}
\label{demthmext7}
\|u_n(t)\|_{L^2(\Omega)}^2
\le\left(\|u_n(T_0)\|_{L^2(\Omega)}^\frac{4(1-m)}{N(1-m)+4}-\lambda MC_\GN^{-\frac4{N(1-m)+4}}(t-T_0)\right)_+^\frac{N(1-m)+4}{2(1-m))},
\end{gather}
for any $n\in\N$ and $t\ge T_0.$ By \eqref{demthmext5} and \eqref{propdep2}, we may pass to the limit in \eqref{demthmext7} and then $u$ satisfies \eqref{thmext11}.
\\
\textbf{Proof of Property~\ref{thmext2}.}
We first note by \eqref{thmext21} that $f\in L^p\big((0,\infty);L^2(\Omega)\big),$ where $p=\frac{2\delta}{2\delta-1}>1.$ By Lemma~\ref{lemden1}, there exist $h\in L^p((0,\infty);\R),$ $(\vphi_n)_{n\in\N}\subset\Dr(\Omega)$ and $(f_n)_{n\in\N}\subset\Dr\big((0,\infty);L^2(\Omega)\big)$ such that, 
\begin{gather}
\label{demthmext8}
\vphi_n\xrightarrow[n\to\infty]{L^2(\Omega)}u_0 \; \text{ and } \; f_n\xrightarrow[n\to\infty]{L^p((0,\infty);L^2(\Omega))}f,	\\
\label{demthmext9}
\text{for a.e.}\, t>0 \; f_n(t)\xrightarrow[n\to\infty]{L^2(\Omega)}f(t),											\\
\label{demthmext10}
\text{for a.e.}\, t>0 \text{ and any } n\in\N, \; \|f_n(t)\|_{L^2(\Omega)}\le h(t).
\end{gather}
For each $n\in\N,$ let $(u_n,f_n)$ be the $H^2$-solution to \eqref{gl}--\eqref{glb} such that $u_n(0)=\vphi_n,$ and set for any $t\ge0,$ $y_n(t)=\|u_n(t)\|_{L^2(\Omega)}^2.$ Let $n\in\N.$ We have by \eqref{L2}, \eqref{demthmext0} and Cauchy-Schwarz' inequality that for a.e.\,$t>0,$
\begin{gather}
\label{demthmext11}
y_n^\p(t)+2\,\alpha\,y_n(t)^\delta\le2\,\|f_n(t)\|_{L^2(\Omega)}\,y_n(t)^\frac12,
\end{gather}
where $\alpha=\min\left\{\cos\theta,\Re(a\e^{\vi\theta})\right\}C_\GN^{-\frac4{N(1-m)+4}}.$ Now, we set
\begin{gather*}
\eps_\star=\min\left\{(2\delta-1)^{-\frac{2\delta-1}{\delta}}(\alpha\delta)^\frac1{1-\delta}(1-\delta)^\frac{2\delta-1}{\delta(1-\delta)},\alpha\,\delta\,(1-\delta)\right\}.
\end{gather*}
Applying Young's inequality to \eqref{demthmext11} we arrive at,
\begin{gather*}
y_n^\p(t)+2\,\alpha\,y_n(t)^\delta\le\frac{2\delta-1}\delta(\alpha\delta)^{-\frac1{2\delta-1}}\|f_n(t)\|_{L^2(\Omega)}^\frac{2\delta}{2\delta-1}+\alpha\,y_n(t)^\delta,
\end{gather*}
for a.e.\,$t>0,$ and then
\begin{gather*}
\text{for a.e.} \; t>0, \; y_n^\p(t)+\alpha y_n(t)^\delta\le g_n(t),
\end{gather*}
where $g_n(t)=\frac{2\delta-1}\delta(\alpha\delta)^{-\frac1{2\delta-1}}\|f_n(t)\|_{L^2(\Omega)}^\frac{2\delta}{2\delta-1}.$ Let $z_n\in W^{1,1}_\loc\big([0,\infty);\R\big),$ with $z_n\ge0$ everywhere in $[0,\infty),$ be the unique solution to
\begin{gather*}
\text{for a.e.} \; t>0, \; z_n^\p(t)+\alpha z_n(t)^\delta=g_n(t),
\end{gather*}
such that $z_n(0)=y_n(0).$ By Lemma~\ref{lemden2}, we have for any $t\ge0,$ $y_n(t)\le z_n(t).$ By \eqref{demthmext9}, \eqref{demthmext10} and the dominated convergence Theorem, we have that $g_n\xrightarrow[n\to\infty]{L^1((0,\infty);\R)}g,$ where for a.e.\,$t>0,$
\begin{gather*}
g(t)=\frac{2\delta-1}\delta(\alpha\delta)^{-\frac1{2\delta-1}}\|f(t)\|_{L^2(\Omega)}^\frac{2\delta}{2\delta-1}.
\end{gather*}
We then infer with the help of \eqref{demthmext8}, \eqref{propdep2} and Lemma~\ref{lemden2} that
\begin{gather}
\label{demthmext12}
\forall t\ge0, \; y(t)\le z(t),
\end{gather}
where $z\in W^{1,1}_\loc\big([0,\infty);\R\big)$ is the unique nonnegative solution to
\begin{gather*}
\text{for a.e.} \; t>0, \; z^\p(t)+\alpha z(t)^\delta=g(t),
\end{gather*}
such that $z(0)=y(0).$ By \eqref{thmext21}, it follows that
\begin{gather*}
\text{for a.e.} \; t>0, \; z^\p(t)+\alpha z(t)^\delta\le z_\star\big(T_0-t\big)_+^\frac\delta{1-\delta},
\end{gather*}
where $z_\star=\big(\alpha\delta^\delta(1-\delta)\big)^\frac1{1-\delta}.$ Finally, let $\zeta_\star=(\alpha\delta(1-\delta)T_0)^\frac1{1-\delta},$ and let $\zeta\in W^{1,1}_\loc\big([0,\infty);\R\big),$ with $\zeta\ge0$ everywhere in $[0,\infty),$ be the unique solution to
\begin{gather*}
\text{for a.e.} \; t>0, \; \zeta^\p(t)+\alpha\zeta(t)^\delta=z_\star\big(T_0-t\big)_+^\frac\delta{1-\delta},
\end{gather*}
such that $\zeta(0)=\zeta_\star.$ By \eqref{thmext21}, we have $z(0)\le\zeta_\star$ and then Lemma~\ref{lemden2} implies that
\begin{gather}
\label{demthmext13}
\forall t\ge0, \; z(t)\le\zeta(t),
\end{gather}
Finally, by the uniqueness of the solution, we obtain
\begin{gather}
\label{demthmext14}
\forall t\ge0, \; \zeta(t)=\zeta_\star T_0^{-\frac1{1-\delta}}\left(T_0-t\right)_+^\frac1{1-\delta}.
\end{gather}
Putting together \eqref{demthmext12}--\eqref{demthmext14}, we get that
\begin{gather*}
\forall t\ge0, \; y(t)\le\zeta_\star T_0^{-\frac1{1-\delta}}\left(T_0-t\right)_+^\frac1{1-\delta},
\end{gather*}
from which the result follows.
\medskip
\end{vproof}

\begin{vproof}{of Proposition~\ref{propm1}.}
By density (in particular \eqref{lemden14}) and continuous dependence \eqref{propdep2}, we may assume $f\in\Dr\big((0,\infty);L^2(\Omega)\big),$ $u_0\in\Dr(\Omega)$ so that $u$ is an $H^2$-solution (\cite[Theorem~2.14]{Part1}). We then have by \eqref{L2},
\begin{gather*}
\forall t\ge T_0, \; \frac12\frac{\d}{\d t}\|u(t)\|_{L^2(\Omega)}^2+\Re(a\e^{\vi\theta})\|u(t)\|_{L^2(\Omega)}^2\le0,
\end{gather*}
from which the result follows.
\end{vproof}

\begin{vproof}{of Theorem~\ref{thm0w}.}
By \eqref{propdep2}, density and Proposition~\ref{propm1}, we may assume that $u_0\in\Dr(\Omega),$ $f\in\Dr\big((0,\infty);L^2(\Omega)\big)$ and $m<1.$ The result then comes from Theorem~\ref{thmext}.
\medskip
\end{vproof}

\baselineskip .4cm


\end{document}